\documentclass[11pt, reqno]{amsart}
\usepackage[latin1]{inputenc}
\usepackage{amssymb}
\usepackage{pdfsync}
\usepackage[english]{babel}
\usepackage{appendix}

\usepackage[pdftex, colorlinks=true,urlcolor=blue,linkcolor=blue,citecolor=blue, linktocpage]{hyperref}
\usepackage{graphicx}

\usepackage[capitalize]{cleveref}
\usepackage{fullpage}
\usepackage{color}
\usepackage{amsmath}
\usepackage{amsfonts}
\usepackage{mathrsfs}
\usepackage{t1enc , graphicx}
\usepackage{verbatim}
\usepackage{bbm}
\usepackage{enumitem}
\usepackage{tikz-cd}
\usepackage{adjustbox}
\usepackage{mathtools}
\usepackage[left= 1.2 in, right= 1.2 in ,top= 1.2 in, bottom = 2.1 in]{geometry}

\newcommand{\Z}{\mathbb{Z}}
\newcommand{\N}{\mathbb{N}}

\renewcommand{\P}{\mathbb{P}}

\newtheorem{theorem}{Theorem}[section]
\newtheorem*{theorem*}{Theorem}

\newtheorem{lemma}[theorem]{Lemma}

\newtheorem{corollary}[theorem]{Corollary}
\newtheorem*{corollary*}{Corollary}

\theoremstyle{definition}
\newtheorem*{definition*}{Definition}
\newtheorem{definition}[theorem]{Definition}

\newtheorem{question}[theorem]{Question}

\theoremstyle{remark}

\newtheorem*{remark*}{Remark}
\newtheorem{remark}{Remark}

\theoremstyle{plain}
\newcounter{MainTheoremCounter}

\newtheorem{Maintheorem}[MainTheoremCounter]{Theorem}

\crefname{MainTheoremCounter}{Theorem}{Theorems}

\Crefname{MainTheoremCounter}{Theorem}{Theorems}

\theoremstyle{plain}
\newcounter{OldTheoremCounter}

\newcommand{\vertiii}[1]{{\left\vert\kern-0.25ex\left\vert\kern-0.25ex\left\vert #1 
		\right\vert\kern-0.25ex\right\vert\kern-0.25ex\right\vert}}

\title{On infinite sumsets and sets of multiple recurrence}

\author{Luke Hetzel}
\address{Department of Mathematics\\
	University of Denver\\
	2390 S. York St, Denver, CO 80210, USA}
\email{luke.hetzel@du.edu}

\usepackage{stix} 
\begin{document}

\begin{abstract}
    We answer two questions of Kra, Moreira, Richter and Robertson regarding the existence of infinite sumsets of the form $B + C$ in dense and sparse sets of integers and the relation of sumsets to sets of recurrence. We then further generalize these results, yielding new characterizations of sets of multiple measurable and topological recurrence.
\end{abstract}
\maketitle

\section{Introduction}
\subsection{Historical Overview}

     For a subset $A \subset \N$ the \emph{upper Banach density of A} is given by 
    \[
        d^{*}(A) = \limsup_{N-M \to \infty} \frac{|A \cap [M, N]|}{N-M}.
    \]
    The celebrated Szemer\'edi theorem \cite{Szemeredi75} states that a set of positive upper Banach density contains arbitrarily long arithmetic progressions. This theorem of Szemer\'edi's and many of its proofs have sparked connections between seemingly unrelated areas of mathematics. Among the various proofs of this theorem, a notable one was due to Furstenberg \cite{Furstenberg_correspondence}.

    An important object that arises from the work of Furstenberg \cite{Furstenberg_correspondence} is the notion of set of measurable recurrence. We call a set $S \subset \N$ a \emph{set of measurable recurrence} (or \emph{set of recurrence} for short) if for every measure preserving system $(X, \mu, T)$ and every set $E$ of positive measure there exists an $n \in S$ with $\mu (E \cap T^{-n}E) > 0$. Here we say $(X, \mu, T)$ is a \emph{measure preserving system} if $(X, \mu)$ is a probability space and $T$ is a measure preserving transformation, i.e. $\mu(T^{-1} E) = \mu(E)$ for every measurable set $E$.
    
    Even though defined as an ergodic theoretic object, sets of recurrence are equivalent to a notion in additive combinatorics called intersective sets. A set $S \subset \N$ is \emph{intersective} if for every set $A \subset \N$ of positive upper Banach density $(A-A) \cap S \ne \varnothing$. Here  $A \pm B$ is given by $A \pm B = \{a \pm b: a \in A, b \in B\}$. It follows from Furstenberg's correspondence principle \cite{Forrest-recurrence-not-strong} and Bergelson's intersectivity lemma \cite{Bergelson_intersectivity} that $S$ is a set of recurrence if and only if it is an intersective set. A corollary of the Poincar\'e Recurrence Theorem is that the set of natural numbers is a set of recurrence. Other examples of sets of recurrence include $\{P(n): n \in \N\}$ where $P \in \Z[x]$ with $P(0) = 0$ and $\P \pm 1$ where $\P$ is the set of primes \cite{Furstenberg_correspondence, Sarkozy78, Sarkozy78b}.

These results opened up a new direction of research in which dynamical methods are used to address combinatorial and number theoretic problems. In particular, this approach is proven to be effective in uncovering combinatorial structure of sets of positive upper Banach density. Along this line, in \cite{Erdos_B+C_sumset}, Moreira, Richter, and Robertson used ergodic methods to show that every set of positive upper Banach density contains a set of the form $B+C$ where $ B, C \subset \N$ are infinite, confirming a long-standing conjecture of Erd\H{o}s. Then, together with Kra, they showed that for every $k \in \N$, every set of positive upper Banach density contains a subset of the form $B + B + t$ for some infinite set $B$ and $t \in \N$ \cite{B+B+t}.

\subsection{Sumsets and Sets of Measurable Recurrence}
    We say a set $S \subset \N$ is a \emph{set of strong recurrence}  if for every measure preserving system $(X, \mu, T)$ and every set $A$ of positive measure we have $\limsup_{n \in S} \mu(A \cap T^{-n}A) > 0$. Every set of strong recurrence is a set of recurrence. However the converse does not hold: Forrest \cite{Forrest-recurrence-not-strong} constructed an explicit example of a set of recurrence which is not a set of strong recurrence.    
    
    In their survey \cite{Kra-Moreira-Richter-Robertson-sumsets-survey}, Kra, Moreira, Richter, and Robertson proved the following connection between sets of strong recurrence and infinite sumsets.

    \begin{theorem}{\cite[Theorem 3.32]{Kra-Moreira-Richter-Robertson-sumsets-survey}}
    \label{thm:kmrr_strong_recurrence}
    Let $S$ be a set of strong recurrence. Then for every $A \subset \N$ with $d^{*}(A) > 0$ there exist infinite sets $B \subset S, C \subset A$ with 
        \begin{equation}\label{eq:B+C<}
        \{b + c : b \in B, c \in C, b< c\} \subset A.
        \end{equation}
\end{theorem}

It is shown in \cite{Kra-Moreira-Richter-Robertson-sumsets-survey} that \cref{thm:kmrr_strong_recurrence} is no longer true if one remove the condition $b < c$ in \eqref{eq:B+C<}. To see this, take $S$ to be the set of squares $\{n^2: n \in \N\}$ and $A = \N \setminus \bigcup_{k =1}^\infty (k^2 - \log(k), k^2 + \log(k))$. Furthermore, it was also pointed out in  \cite{Kra-Moreira-Richter-Robertson-sumsets-survey} that any set $S$ that satisfies the conclusion of \cref{thm:kmrr_strong_recurrence} must be a set of recurrence. This observation led Kra, Moreira, Richter and Robertson to pose a natural question: is strong recurrence assumption necessary for \cref{thm:kmrr_strong_recurrence} or could it be replaced with the weaker notion -- set of recurrence?

\begin{question}\cite[Question 3.34]{Kra-Moreira-Richter-Robertson-sumsets-survey}\label{ques:kmmr rec}
    Is it true that if $S$ is a set of recurrence and $A \subset \N$ such that $d^{*}(A) > 0$ then there exist infinite sets $B \subset S, C \subset A$ with 
        \[
        \{b + c: b \in B, c \in C, b < c \} \subset A ?
        \]
\end{question}

In \cref{sec_meas} we show that not only the answer to \cref{ques:kmmr rec} is positive, but also it follows from a more general theorem below regarding sets of multiple recurrence.

\begin{definition}
    A set $ S \subset \N$  is called a \emph{set of $k$-recurrence} if for every measure preserving system $(X, \mu, T)$ and for every $E \subset X$ with $\mu(E) > 0$, there exists an $n \in S$ such that,
    $$\mu(E \cap T^{-n}E \cap T^{-2n}E \cap \ldots \cap T^{-kn}E) >0.$$
\end{definition}

The notion of $k$-recurrence arises from the aforementioned work of Furstenberg \cite{Furstenberg_correspondence}. It is immediate that $1$-recurrence is the same as measurable recurrence.
Likewise, on the combinatorial side, the generalization of intersective sets is $k$-intersective sets. A set $S \subset \N$ is \emph{$k$-intersective} if for any $A \subset \N$ of positive upper Banach density, there exist $n \in S$ such that
\[
    A \cap A - n \cap \cdots \cap A - kn \neq \varnothing, 
\]
i.e. there is $a \in A, n \in S$ satisfying
\[
    a, a + n, \ldots, a + kn \in A.
\]
Similarly to the relation between measurable recurrence and intersectivity, a set $S$ is a set of $k$-recurrence if and only if it is $k$-intersective \cite{Bergelson_intersectivity, Furstenberg_correspondence}. This equivalence together with the fact that $\N$ is a set of $k$-recurrence for all $k \in \N$ \cite{Furstenberg_correspondence} implies Szemer\'edi's theorem.

All examples of sets of recurrence we mentioned before happen to be sets of $k$-recurrence for all $k$. However, this phenomenon is not true in general. Furstenberg \cite{recurrence_in_erg_theory_fur} constructed a set of $1$-recurrence which is not a set of $2$-recurrence. It is shown by Frantzikinakis, Lesigne, and Wierdl \cite{Frantzikinakis-Lesigne-Wierdl-k-not-k+1-recurrence} that for every $k$, there is a set of $k$-recurrence which is not a set of $(k+1)$-recurrence. 

Our first theorem shows that $k$-recurrence is equivalent to a seemingly stronger notion than $k$-intersectivity.

\begin{Maintheorem}\label{measurable}
    Let $S \subset \N$ and $k \in \N$. The following are equivalent.
    \begin{enumerate}
        \item \label{item:measurable_1} $S$ is a set of $k$-recurrence. 
        \item \label{item:measurable_2} For every set $A \subset \N$ with $d^{*}(A) > 0$, there exist infinite $B \subset S$ and $C \subset A$ such that 
        \[
         \{ib + c : b \in B, c \in C, b<c , 0 \leq i \leq k \} \subset A.
        \]
        \item \label{item:measurable_3} For every set $A \subset \N$ with $d^{*}(A) > 0$ there exist infinite $B \subset S$ and $C \subset A$ such that for every $m \in \N$ we have 
        \[
         \{ i_1 b_1 + i_2 b_2 + \ldots +i_m b_m + c: b_j \in B, c \in C,  b_1 < b_2 < \ldots < b_m < c, 0 \leq i_j \leq k, 1 \leq j \leq m \} \subset A.
        \]
    \end{enumerate}
\end{Maintheorem}

The definition of $k$-intersectivity is recovered if we take $B$ and $C$ in \eqref{item:measurable_2} to be singletons. To illustrate the configurations appearing in \eqref{item:measurable_3}, we provide here some examples. If $m=2, k = 1$, \eqref{item:measurable_3} gives that $A$ contains
\[
    \{ b_1 + c, b_1 + b_2 + c: b_1, b_2 \in B, c \in C, b_1< b_2 < c \} .
\]
If $m = 2, k = 2$, $A$ contains
\[
    \{b_1 + c, 2b_1 + c, b_1 + b_2 +c,2b_1 + b_2 + c, b_1 + 2b_2 +c, 2b_1 + 2b_2 +c : b_1, b_2 \in B, c \in C, b_1< b_2<c \}
\]
In particular, by taking $k = 1$ in \cref{measurable}, the set $A$ would contain $\{b + c: b \in B, c \in C, b < c\}$, yielding the following corollary and positive answer to \cref{ques:kmmr rec}.

\begin{corollary}\label{3.34 answer}
    A set $S \subset \N$ is a set of measurable recurrence if and only if for every $A \subset \N$ with $d^{*}(A) > 0$ there exist infinite subsets $B \subset S, C \subset A$ such that
        \[
        \{b + c : b \in B, c \in C, b< c \} \subset A.
        \]
    
\end{corollary}

\subsection{Sumsets and Sets of Topological Recurrence}
In this section, we turn our attention to the relationship between infinite sumsets and sets of topological recurrence. A \emph{topological dynamical system} is a pair $(X, T)$ where $X$ is a compact metric space and $T: X \to X$ is a continuous map. The system $(X, T)$ is \emph{minimal} if it contains no non-trivial subsystems, or equivalently, if the orbit of every $x \in X$, $\{T^n x: x \in \N \cup \{0\}\}$ is dense in $X$.

The topological counterparts of sets measurable recurrence are sets of topological recurrence. We say $S \subset \N$ is \emph{set of topological recurrence} if for every minimal system $(X,T)$, and every nonempty open set $U \subset X$, there exists an  $n \in S$ such that $U \cap T^{-n}U \ne \varnothing$. More generally, we have the notion of \emph{sets of topological $k$-recurrence}.

\begin{definition}
    $S \subset \N$ is called a \emph{set of topological $k$-recurrence} if for every minimal system $(X, T)$, and every nonempty open set $U \subset X$, there exists $n \in S$ such that $U \cap T^{-n} U \cap \ldots \cap T^{-kn} U \neq \varnothing$.
\end{definition}
 
Since every minimal system admits an invariant measure of full support, every set of $k$-recurrence is a set of topological $k$-recurrence. An example of Kriz \cite{Krisz_example}  shows the converse is false. Parallel to the equivalence between $k$-recurrence and $k$-intersectivity, a set $S \subseteq \N$ is a set of topological $k$-recurrence if and only if it is \emph{chromatically $k$-intersective}, namely for every finite coloring $\N = \bigcup_{i=1}^{\ell} A_i$, there exists a color $A_i$ and an $n \in S$ such that $A_i \cap (A_i - n) \cap \ldots \cap (A_i - kn) \neq \varnothing$ (see \cite{McCutcheon-elemental-methods-ergodic}).

    We say a set $S\subset \N$ is \emph{syndetic} if there exists $ n \in \N$ such that $\bigcup_{i = 0}^n (S-i) \supset \N$. A set $T \subset \N$ is \emph{thick} if for every $n \in \N$ there exists a $t \in T$ with $t, t+1 , t+2, \ldots t+n \in T$. A set $A \subset \N$ is \emph{piecewise syndetic} if it is the intersection of a syndetic set and a thick set. It is well-known (for example, see \cite{McCutcheon-elemental-methods-ergodic}) that $S$ is chromatically $k$-intersective if and only if for every piecewise syndetic set $A$, there exists $n \in S$ such that 
    \[
        A \cap (A - n) \cap \cdots \cap (A - kn) \neq \varnothing.
    \]

    Sets of topological recurrence provide a similar bridge between problems about colorings and topological dynamics as sets of measurable recurrence do between problems about density and measurable dynamics. As such, one might ask if we can achieve a similar result to \cref{measurable} in topological setting. In \cref{sec_top} we show it is possible and have the following theorem.

    \begin{Maintheorem}\label{topological}
    Let $S \subset \N$ and $k \in \N$. The following are equivalent.
        \begin{enumerate}
            \item \label{item:topological_1} $S$ is a set of topological $k$-recurrence.
            \item \label{item:topological_2} For every finite coloring $\N = \bigcup_{i=1}^\ell A_i$ there exist $j$, $1 \leq j \leq \ell$ and infinite $B \subset S, C \subset A_j$ with 
                \[
                \{ib + c: b \in B, c \in C, b < c, 0 \leq i \leq k \} \subset A_j.
                \]
            \item \label{item:topological_3} For every piecewise syndetic set $A \subset \N$ there exist infinite $B \subset S, C \subset A$ with 
                \[
                 \{ib + c: b \in B, c \in C, b < c, 0 \leq i \leq k \} \subset A.
                \]
            \item \label{item:topological_4} For every piecewise syndetic set $A \subset \N$ there exists infinite $B \subset S, C \subset A$ with for every $m \in \N$ we have   
            \[
                \{ i_1 b_1 + i_2 b_2 + \ldots + i_m b_m + c : b_j \in B, c \in C, b_1 < \ldots < b_m < c, 0 \leq i_j \leq k, 1 \leq j \leq m\} \subset A.
            \]
        \end{enumerate}
\end{Maintheorem}

\subsection{Sumsets in Sets of Zero Banach Density}

    Many combinatorial properties of $\N$ not only carry over to dense sets (sets of positive upper Banach density) but also to sufficiently nice sparse sets (sets of zero Banach density).
    In \cite{Green_Roththeoremprime}, Green devised a transference principle to deduce from Roth's theorem \cite{Roth-theorem} that every set which is dense relative to the set of primes, $\P$, contains three-term arithmetic progressions. This transference principle was a precursor to another one which enabled Green and Tao \cite{Green-Tao_thm} to prove that every dense subset of the primes contains arbitrarily long arithmetic progressions.

    Since then, many variants of the transference principle have been devised	to prove combinatorial theorems in sparse sets of integers such as the	squares \cite{broPre} or the sums of two squares \cite{matt}. Given this backdrop, it was asked in \cite{Kra-Moreira-Richter-Robertson-sumsets-survey} whether the phenomenon of infinite sumsets could be transferred from dense subsets of $\N$ to the set of primes or other sufficiently nice sets of zero Banach density.

    For $A, S \subset \N$, we define the \emph{lower relative density of $A$ in $S$} to be 
        \[
        \underline{d}_S (A) = \liminf_{N \to \infty} \frac{|A \cap S \cap [N]|}{|S \cap [N]|},
        \]
    where $[N] = \{1, 2, \ldots, N \}$. In the case where $S = \N$, lower relative density is simply called \emph{lower density}.

    \begin{question}\label{KMRR4_7}{\cite[Question 4.7]{Kra-Moreira-Richter-Robertson-sumsets-survey}}
    Is there a set $F \subset \N$ of zero upper Banach density such that for every $A \subset F$ with $\underline{d}_F(A) > 0$, there exist infinite $B, C \subset \N$ satisfying $B + C \subset A$?
\end{question}

    A negative answer to \cref{KMRR4_7} would indicate that the largeness of the ambient set $F$ is essential for the containment of an infinite sumset $B + C$. On the other hand, if the answer is positive, it will imply the existence of this structure does not depend on the size of $F$ but possibly on the regularity of its composition. In the same paper, it was observed that a positive answer to the next question will give a negative answer to \cref{KMRR4_7}.

    \begin{question}{\cite[Question 4.8]{Kra-Moreira-Richter-Robertson-sumsets-survey}}\label{0densityquestion}
        Is it true that if $F \subset \N$ has zero upper Banach density, there exists a subset $A \subset F$ with $\underline{d}_F(A) > 0$ such that for all but finitely many $t \in \N$ we have
    \[
    \underline{d}_F (A) = \underline{d}_F (A \setminus (A- t )) ?
    \]

\end{question}

The answer to \cref{0densityquestion} turns out to be negative and this is a corollary of our next theorem.

\begin{Maintheorem}\label{Y_thing}
    There exists a set $F$ of zero Banach density such that for every $A \subset F$ with $\underline{d}_F(A) > 0$ and every sequence of natural numbers $(N_t)_t$, and every $M \in \N$, 
    \[
        Y = A \setminus \bigcup_{t \geq M} \left(  (A - t) \cap [N_t, \infty) \right)
    \]
    has zero relative density in $F$. 
\end{Maintheorem}

\begin{corollary}\label{0density_theorem}
    There exists a set $F \subset \N$ of zero upper Banach density such that for every set $A \subset F$ with $\underline{d}_F(A) > 0$ there exist infinitely many $t \in \N$ for which $\underline{d}_F (A \setminus (A-t)) < \underline{d}_F (A)$. 
\end{corollary}
 
 As \cref{0densityquestion} has a negative answer, it has no implication on \cref{KMRR4_7} and so the latter is still open.

\textbf{Acknowledgments.}
We thank Anh Le for his valuable guidance and feedback. We also thank Ethan Ackelsberg, Evans Hedges, Bryna Kra, Ronnie Pavlov, Florian Richter, Alex Sedlmayr and Casey Schlortt for their insight and helpful conversations.

\section{Sets of Measurable Recurrence}\label{sec_meas}
    In this section we prove \cref{measurable}. In order to do this we first prove a technical lemma. 

\begin{lemma}\label{lem:krecurrence}
    Let $S$ be a set of $k$-recurrence. Then for every measure preserving system $(X, \mu, T)$ and for every $E \subset X$ with $\mu(E) >0$, there exists an infinite sequence $t_1 < t_2 < \ldots  \in S$ such that for each $ m \in \N$, if $F_m = \{ i_1 t_{j_1} + i_2 t_{j_2} + \ldots + i_m t_{j_m} : 0 \leq i_l \leq k, j_1 < j_2 < \ldots < j_m \}$ then
        \[
        \mu \left( E \cap \bigcap_{n \in F_m} T^{-n}E \right) >0.
        \]
\end{lemma}
\begin{proof}
    Let $S$ be a set of $k$-recurrence. Let $(X, \mu, T)$ be a measure preserving system and $E \subset X$ with $\mu(E) > 0$. 
    Since $S$ is a set of $k$-recurrence, there exists $t_1 \in S$ with  
    \[
        \mu \left(E \cap \bigcap_{i=1}^{k}T^{-it_1}E\right)>0.
    \]
    Take $E_1 = E \cap \bigcap_{i=1}^{k}T^{-it_1}E$. Since $S$ is a set of $k$-recurrence we know that $S \setminus [0, kt_1]$ is a set of $k$-recurrence and so there exists $t_2 > kt_1$ with \\
        $\begin{matrix*}[l]
        0 < \mu \left( E_1 \cap \bigcap_{j=1}^{k}T^{-jt_2}E_1 \right) & = & \mu\left(E \cap \bigcap_{i=1}^{k}T^{-it_1}E \cap \bigcap_{j=1}^{k}T^{-jt_2}\left(E \cap \bigcap_{i=1}^{k}T^{-it_1}E\right)\right) \\[1em]
        & = & \mu \left(E \cap \bigcap_{i=1}^{k}T^{-it_1}E \cap \bigcap_{i=1}^{k}T^{-it_2} E  \cap \bigcap_{i=1}^k \bigcap_{j = 1}^k T^{-i t_1 -j t_2}E\right). \\
        \end{matrix*}$ \\

    Take $E_2 = E \cap \bigcap_{i=1}^{k}T^{-it_1}E \cap \bigcap_{i=1}^{k}T^{-it_2} E  \cap \bigcap_{i, j =1}^k T^{-i t_1 -j t_2}E$ and proceed by induction.
\end{proof}
With this, we now begin our proof of \cref{measurable}.

\begin{proof}[Proof of \cref{measurable}]
    Restricting \eqref{item:measurable_3} to the case where $m = 1$ we achieve \eqref{item:measurable_2}. We now show that \eqref{item:measurable_2} implies \eqref{item:measurable_1}. The case $k = 1$ of this direction is mentioned without proof in \cite{Kra-Moreira-Richter-Robertson-sumsets-survey}. We now provide the proof for arbitrary $k$. 
    
    Suppose $S$ has property \eqref{item:measurable_2}. Let $A \subset \N$ such that $d^{*}(A) >0$. Then by hypothesis there exists infinite sets $B \subset R, C \subset A$ with
        \[
        \bigcup_{i = 1}^k \{ib + c : b \in B, c \in C, b<c \} \subset A.
        \]
    Choose $b \in B, c \in C, b< c$. Then for each $i \in [k]$ we have that $c + bi \in A$. So $c \in A \cap A-b \cap A-2b \cap \ldots \cap A-kb \ne \varnothing$. Since $A$ is an arbitrary set of positive upper Banach density, $S$ is a $k$-intersective and is therefore a set of $k$-recurrence.

    We will now show \eqref{item:measurable_1} implies \eqref{item:measurable_3}. Let $S$ be a set of $k$-recurrence and $A \subset \N$ with $d^{*}(A) > 0$. Then by Furstenberg's correspondence principle \cite{Furstenberg_correspondence} there exists a measure preserving system $(X, \mu, T)$ and a set $E \subset X$ with $\mu(E) \geq d^{*}(A)$ such that for every finite set $F \subset \N \cup \{0\}$
    \[
        d^{*}\left( \bigcap_{n \in F}A-n\right) \geq \mu \left( \bigcap_{n \in F} T^{-n}E \right).
    \]
      Now by \cref{lem:krecurrence} there exists an infinite increasing sequence $t_1 < t_2 <.... \in S$ such that for each $ m \in \N$, if $F_m = \{ i_1 t_{j_1} + i_2 t_{j_2} + \ldots + i_m t_{j_m} : 1 \leq i_l \leq k, j_1 < j_2 < \ldots < j_m \}$ then
        \[
        \mu \left(E \cap \bigcap_{n \in F_m} T^{-n}E\right) >0.
        \]
    Now take $b_1 := t_1$ and we have that
    \[
        d^{*}\left(A \cap \bigcap_{i = 1}^{k}\left( A-ib_1 \right) \right) > 0.
    \]
    Since $A \cap  \bigcap_{i = 1}^{k}\left(A-ib_1 \right)$ has positive density, it is infinite. Pick $c_1 \in A \cap \bigcap_{i = 1}^{k}(A-ib_1 )$ such that $c_1 > kb_1 $. Take $b_2 = \min\{t_n: t_n > c_1 \} $ and proceed by induction to choose infinite $B \subset S$ and $C \subset A$ that satisfy \eqref{item:measurable_3}.
 
\end{proof}

\section{Sets of Topological Recurrence}\label{sec_top}

In this section, it is more convenient for us to work with the set of nonnegative integers $\N_0 = \N \cup \{0\}.$ %This is allows for clarity of indexing. 
We now begin our proof of Theorem \ref{topological}.

\begin{proof}[Proof of \cref{topological}]
It is well known that in any coloring of $\N_{0}$, there is a color which is piecewise syndetic. Therefore we have \eqref{item:topological_3} implies \eqref{item:topological_2}. By restricting to the case where $m=1$, \eqref{item:topological_4} implies \eqref{item:topological_3}. It remains to show that \eqref{item:topological_1} implies \eqref{item:topological_4} and that \eqref{item:topological_2} implies \eqref{item:topological_1}. 

We will first show \eqref{item:topological_2} implies \eqref{item:topological_1}. Suppose that $S$ has property \eqref{item:topological_2}. We will show that $S$ is chromatically $k$-intersective. 
    Let $\N_{0} = \bigcup_{i = 1}^t A_i$ be a finite coloring of $\N_{0}$. By hypothesis there exists $j \in [t ]$ and an infinite $B \subset S, C \subset A_j$ with 
       \[
           \{ib + c: b \in B, c \in C, b < c, 0 \leq i \leq k \} \subset A_j.
        \]
    Choose $b \in B, c \in C$ with $b<c$. Then for each $i \in [k]$ we have $c + bi \in A_j$ and so $c \in A_j \cap A_j -b \cap \ldots \cap A_j - ib$. In particular, the intersection is nonempty. Since the coloring $\N_{0} = \bigcup_{i=1}^{\ell} A_i$ is arbitrary, $S$ is chromatically $k$-intersective.

We now show that \eqref{item:topological_1} implies \eqref{item:topological_4}. Let $S$ be a set of topological $k$-recurrence. Let $A = P \cap T$ with $P$ syndetic and $T$ thick. Let $x \in \{0,1\}^{\N_{0}}$ be the characteristic function for $A$. For an element $a \in \{0,1\}^{\N_{0}}$ we denote $a_i = a(i)$ to be the $i$-th term in the sequence $a$. Take $\sigma : \{0,1\}^{\N_{0}} \to \{0,1\}^{\N_{0}}$ to be the left shift operator given by $\sigma(a)_i = a_{i+1}$. Define a metric $d: \{0,1\}^{\N_{0}} \times \{0,1\}^{\N_{0}} \to [0,\infty)$ by $d(a,b) = \sum_{i=0}^{\infty} \frac{|a_i -b_i|}{2^{i+1}}$. Note that $\{0,1\}^{\N_{0}}$ with $d$ is a compact metric space and $\sigma$ is continuous. 

For a point $a \in \{0,1\}^{\N_{0}},$ we define $O(a) = \{ T^n a : n \in \N_{0} \}$, the orbit of $a$, and $\overline{O(a)}$ to be the closure of $O(a)$ in $(\{0,1\}^\N_{0} , d)$. Let $X = \overline{O(x)}$. We will first show that there exists $ y \in X$ such that $y$ is the characteristic function for a syndetic set.

    Choose the smallest $n \in \N$ such that $P \cup P-1 \cup \ldots \cup P-n-1 \supseteq \N$. Let $m \in \N$ and consider $A \cap [m, m+n]$. We know that there are infinitely many $m \in \N$ such that $m + \{0, 1, \ldots n \} \subset T$, i.e. $m$ is the beginning of an interval included in $T$. 

    For such $m$, $|A \cap [m, m+n]| \geq 2$. Now note that there are only finitely many ways to choose at least two slots from among $n+1$ slots. By the pigeonhole principle there exists $a_1, a_2$ such that for $m_i + a_1, m_i+a_2 \in A$ for a set of $m_i$ such that $m_i$'s are the beginning of intervals included in $T$ and the lengths of such intervals tend to infinity.

    Likewise, from among such pairs $m_i + a_1, m_i + a_2$ there must exist an $a_3 > a_2$ with $a_3 -a_2 < n$ such that $m_{i_k} + a_1, m_{i_k}+a_2, m_{i_k}+ a_3 \in A$ for a set of $m_{i_k}$ such that $m_{i_k}$'s are the beginning of intervals included in $T$ and the lengths of such intervals tend to infinity. Carrying this argument indefinitely, we build a sequence $y$ that appears in $\overline{O(x)}$ such that the distance between consecutive $1$s in $y$ is less that or equal to $n$. Therefore $y$ is the characteristic function for a syndetic set.

    Let $Y$ be a minimal subsystem of $\overline{O(y)}$. Further, suppose towards a contradiction that for every $z \in Y$ we have $z_0 = 0$. Then for every $z \in Y$ we have $\sigma^n (z)_0 = 0$ and so $Y = \{\hat{0}\}$ where $\hat{0}$ is the sequence of all 0s. Therefore $0 \in \overline{O(y)}$. However this is impossible since $y$ is the characteristic function of a syndetic set. Therefore there exists $z \in Y$ with $z_0 = 1$. 

    Let $\epsilon = \frac{1}{4}$. Let $B(z, \epsilon)$ be the ball in $Y$ of radius $\epsilon$ about $z$. Note that $\epsilon$ is chosen so that $w \in B(z, \epsilon)$ implies $w_0 = z_0$. Since $S$ is a set of topological $k$-recurrence we have that there exists $ b_1 \in R$ with 
    \begin{equation}
    \label{eq:define_B_1}
        B_1 = B(z, \epsilon) \cap \bigcap_{i=1}^k \sigma^{-ib_1}B(z, \epsilon ) \ne \varnothing.
    \end{equation}
    Take $w \in B_1$. It follows that
    \[
    z_0 =w_0=w_{b_1}= \ldots = w_{kb_1}.
    \]

    Since $w \in \overline{O(x)}$ we have that there exists $ c_1 \in \N_{0}$ with $w_i = \sigma^{c_1}(x)_i$ for every $i, 0 \leq i \leq kb_1$. Since $\sigma^{c_1}(x)_0 = w_0 = z_0 =1$ we know that $c_1 \in A$. Also $c_1 + ib_1 \in A$ for every $i, 1 \leq i \leq k$. We claim that  we can choose such $c_1$ that $c_1 > b_1$.

    If $w$ is a limit point of $O(x)$, the orbit of $x$ under $\sigma$, then we are done. This is since there would be infinitely many $n$ with $d(\sigma^n (w),x ) < \frac{1}{2^{kb_1 + 1}}$ and we can therefore choose $c_1$  to be such an $n$ larger than $b_1$. Instead suppose $w \in O(x)$ and not a limit point. Since $w \in Y$ and $Y$ is minimal we know that $w$ is recurrent in $Y$. So there exists an $ n > kb_1$ such that 
        \[
        d\left(w, \sigma^{n}(w)\right) < \frac{1}{2^{kb_1 + 1}}.
        \]
    Now $w \in O(x)$ so there exists $ n_1 \in \N_{0}$ such that $\sigma^{n_1}(x) =w$. Therefore for every $i, 0 \leq i \leq kb_1$
        \[
        \sigma^{n+n_1}(x)_i = w_i.
        \]
    Choose $c_1 = n+ n_1$. Now since $w_0=1$ we have $\sigma^{c_1}(x)_0 = 1$ and therefore $c_1 \in A$. Likewise since $w_{b_1}, w_{2b_1}, \ldots w_{kb_1} = 1$ we have that 
        \[
        c_1 + b_1, c_1 + 2b_1 , \ldots c_1 + kb_1 \in A
        \]
        and note that $b_1 \in R, c_1 \in A$ and $b_1 < c_1$.

    Now since $B_1 = B(z, \epsilon) \cap \bigcap_{i=1}^k \sigma^{-ib_1}B(z, \epsilon ) \ne \varnothing$, and $S$ is a set of $k$-topological recurrence, we can find a $b_2 \in S$ such that $b_2 > c_1$ and 
        \[
        B_2 = B_1 \cap \bigcap_{i=1}^k \sigma^{-ib_2} B_1 \ne \varnothing.
        \]
   
    With $B_2$ in the place of $B_1$ (as defined in \eqref{eq:define_B_1}), we repeat the same argument as above. In general, by induction, we can find sequences $(b_n), (c_n)$ such that $b_n \in R$, $c_n \in A$ with $b_1 < c_1 < b_2 < c_2 < \ldots$ and when we take $B = \{b_n : n \in \N \}$ and $C = \{c_n: n \in \N \}$ we get for every $m \in \N$
    \[
        \{ i_1 b_1 + i_2 b_2 + \ldots + i_m b_m + c : b_j \in , c \in C, 0 \leq i_j \leq k \text{ for every } j, 1 \leq j \leq m, b_1 < \ldots < b_m < c \} \subset A
    \]
    which satisfies \eqref{item:topological_4}.

\end{proof}

\section{Sets of Zero Banach Density}\label{sec_0density}

To explain why \cref{Y_thing} implies \cref{0density_theorem} and why a positive answer to \cref{0densityquestion} would give a negative answer to \cref{KMRR4_7}, we present the following lemma. This lemma was mentioned without proof by Kra, Moreira, Richter and Robertson in \cite{Kra-Moreira-Richter-Robertson-sumsets-survey}.

\begin{lemma}\label{4_8implies_not4_7}
If \cref{0densityquestion} had a positive answer, then \cref{Y_thing} would be false and \cref{KMRR4_7} would have a negative answer, i.e. for any set $F \subset \N$ of zero Banach density one may find a subset $A \subset F$ with $\underline{d}_F (A) > 0$, $M \in \N$ and a sequence $(N_t)_{t \geq M}$ such that 
\[
    Y = A \setminus \bigcup_{t \geq M} \left(  (A - t) \cap [N_t, \infty) \right)
\]
has positive relative density in $F$, but does not contain any set of the form $B+C$ where $B, C$ are infinite subsets of $\N$. 
\end{lemma}

    \begin{proof}%[Proof of \cref{4_8implies_not4_7}]
    
Suppose that for every set $F \subset \N$ with zero upper Banach density, there exists a subset $A \subset F$ with $\underline{d}_{F}(A) > 0 $ such that for all but finitely many $t \in N$, we have $\underline{d}_{F}(A) = \underline{d}_{F}(A \setminus (A-t))$. Let $M $ be greater than all the $t$ for which $\underline{d}_{F}(A) > \underline{d}_{F}(A \setminus (A-t))$.

 Now pick $t_1 = M$. Since $\delta = \underline{d}_F (A) = \underline{d}_F (A \setminus (A-t_1) )$, there exists $ N_1 \in \N$ such that for every $n \geq N_1$,
        \[
        \frac{|F \cap (A \setminus (A - t_1) ) \cap [n]| }{|F \cap [n]|} \geq \delta - \frac{\delta}{4}.
        \]
    Take $Y_1 = A \setminus (A-t_1 \cap [N_t, \infty))$ and take $t_2 = t_1 + 1$. 

    Likewise there exists $N_2 > N_1$ such that for every $n \geq N_2$,
        \[
        \frac{|F \cap (A \setminus A - t_2 ) \cap [n]| }{|F \cap [n]|} \geq \delta - \frac{\delta}{16}.
        \]
    
    Take $Y_2 = A \setminus \bigcup_{i=1}^{2} (A - t_i \cap [N_i, \infty))$ and note that $Y_1 \cap [N_1] = Y_2 \cap [N_1]$ and that $\underline{d}_F (Y_2) \geq \delta - \sum_{i=1}^2 \frac{\delta}{4^i}$.

    Inductively, for $t > M$, there exists $N_{t} > N_{t-1}$ such that for $n \geq N_t$
    \[
        \frac{|F \cap (A \setminus A - t ) \cap [n]| }{|F \cap [n]|} \geq \delta - \frac{\delta}{4^t}.
    \]
    Define 
    \[
        Y = A \setminus \bigcup_{t \geq M} \{a \geq N_t : a \in A-t \}.
    \]
    We will show that $\underline{d}_F (Y) \geq \frac{2 \delta}{3} = \delta - \sum_{i=1}^{\infty} \frac{\delta}{4^i}$.
    
    Suppose towards a contradiction that $\underline{d}_F (Y) < \frac{2 \delta}{3} = \delta - \sum_{i=1}^{\infty} \frac{\delta}{4^i}$. Then there exists a sequence $(n_i)_i \to \infty$ and an $\epsilon > 0$ such that for every $i$
        \[
        \frac{|F \cap Y \cap [n_i]|}{|F \cap [n_i]|} < \frac{2\delta}{3} - \epsilon.
        \]
    Pick $n_i > N_1$. There exists $m \in \N$ with $n_i \in \{N_m, N_{m+1}]$. But by construction $Y \cap [N_{m+1}] = Y_m \cap [N_{m+1}]$ and so 
        \[
        \frac{|F \cap Y_m \cap [n_i] |}{| F \cap [n_i]|} \geq \delta - \sum_{i=1}^{m} \frac{\delta}{4^i} > \delta - \sum_{i=1}^{\infty} \frac{\delta}{4^i} - \epsilon = \frac{2\delta}{3} -\epsilon. 
        \]
    This gives the desired contradiction and therefore $\underline{d}_F (Y) \geq \frac{2 \delta}{3} > 0$.
    
    Finally let $B, C \subset \N$ be arbitrary infinite sets. There exist $c_1, c_2 \in C$ such that the $c_2 - c_1 > M$. It follows that $B + C \supset B + \{c_1, c_2\}$ contains infinitely pairs whose difference is $c_2 - c_1 > M$. From its construction, $Y$ cannot contain infinitely many such pairs and so cannot contain $B + C$.

    \end{proof}

For a non-decreasing sequence of natural numbers $\{x_n\}_{n \in \N}$, we define $FS(\{x_n\}) = \{ \sum_{i=1}^{k} x_{n_i}: x_{n_i}$ is a subsequence of $x_n \}$. A set of the form $FS(\{x_n\})$ is called an \emph{IP set} in the literature. We have a lemma.

\begin{lemma}\label{S_has_0_density}
    The set $FS(\{4^n\}_{n \geq 1})$ has zero Banach density. 
\end{lemma}
\begin{proof}
Let $A = FS(\{4^n\}_{n \geq 1})$.
    First note that all elements of $A$ are of the form $x = 4^{i_1} + 4^{i_2} + \ldots 4^{i_k}$ for some increasing set of numbers $i_1 < i_2 < \ldots < i_k$. Also note that if $x = 4^{i_1} + 4^{i_2} + \ldots 4^{i_k}$ then $|[1, x] \cap F| = 2^{i_1 -1} + 2^{i_2 -1} + \ldots + 2^{i_k -1}$. Let $\epsilon > 0$. We will find an $N \in \N$ such that for all $n \geq N$ and all $M \in \N$, we have $\frac{|[M, M+ n] \cap A|}{n} < \epsilon.$ Choose $a$ such that $(\frac{1}{2})^a < \epsilon$ and $N > 4^a$. Let $n \geq N$ and let $b$ be the largest integer such that $4^b \leq n$.
    
    Let $M \in \N$. We consider the cardinality of $[M, M+n] \cap A$. Without loss of generality assume both $M, M + n \in A$ (If they are not we can reduce the size of $n$ and get a larger ratio). 
    Suppose $|[1,M]\cap A| = 2^{i_m -1} + \ldots + 2^{i_1 -1}$ and $|[1, M+n] \cap A| = 2^{j_p -1} + \ldots + 2^{j_1 -1}$. Since $j_p \leq b$, it follows that 
    \[
        |[M, M+n] \cap A| = 1 +  2^{j_p -1} + \ldots + 2^{j_1 -1}- (2^{i_m -1} + \ldots + 2^{i_1 -1}) \leq 2^b.
    \]
Therefore,
    \[
        \frac{|[M, M+n] \cap A|}{n} \leq \frac{2^b}{4^b} \leq \frac{1}{2^a} < \epsilon.
    \]
\end{proof}

\begin{proof}[Proof of \cref{Y_thing}]
     Take $S = FS(\{4^n\}_{n>0}) \cup \{1\}$. First, by \cref{S_has_0_density}, $S$ has zero upper Banach density.
     The sequence $S' = 1_S$ can be constructed as a limit of words on a $0,1$ alphabet as follows.
     Let $W_1 = 1001$, $W_2 = 1001000000000001001 =W_1 0^{11}W_1$ and inductively $W_{k+1} = W_k 0^{(4^{k+1}- \sum_{i=1}^k 4^i)}W_k$. Note that $S'|_{[1,\sum_{i=1}^k 4^i]} = W_k$.

    We say that $n$ appears as a distance in $W_k$ if there exist two 1s appearing as the $i$th and $(i+n)$th letters in $W_k$. Note that the construction of $S'$ guarantees that the distance between two copies of $W_k $ in $S'$ is never less that the distance between any two copies of $W_k$ in $W_{k+1}$. Furthermore the distance between a copy of $W_k$ and any other $1$ is never less that the distance between two copies of $W_k$ in $W_{k+1}$. 
    
    Let $A \subset S$ have positive relative lower density, let $M \in \N$, and let $(N_t)_t$ be a sequence of natural numbers. Let $\ell$ be the smallest integer such that the length of $W_\ell$ is larger than $M$. Define
    \[
        Y := A \setminus \bigcup_{t \geq M} \left(  (A - t) \cap [N_t, \infty) \right).
    \]
    Let $N \in \N$ with $N > \ell+1$. Pick $N^*$ to be $\sum_{n=1}^{N+2} 4^n$, the length $W_{N + 2}$ and $L = \max \{ N_t: 1 \leq t \leq N^* \}$. 

    Note that we may view $S'$ as an infinite alternating concatenation of $W_{N+2}$ and blocks of $0^k$ for various $k \in \N$, i.e.
        \[
        S'=W_{N+2}0^{k_1}W_{N+2}0^{k_2} W_{N+2} 0^{k_3} \ldots
        \]
    Consider a copy of $W_{N+2}$ which appears after the $L$th copy of $W_{N+2}$. Note that such a copy appears after $L$, and so all distances greater than $M$ appearing in $W_{N+2}$ have been removed from $1_Y$.  Therefore, at most $2^{\ell+1}$ digits $1$ of that copy of $W_{N+2}$ remain in $1_Y$. On the other hand, there are $2^{N+2}$ digits $1$ of that copy are in $1_S$. From this we see
        \[
        \liminf_{n \to \infty} \frac{|Y \cap S \cap [n]|}{|S \cap [n]|} 
        < \frac{2^{\ell + 1}}{2^N}
        \]
    for every $N \in \N$ and so the left hand side is zero.
\end{proof}

\begin{remark}
The proof of \cref{Y_thing} remains unchanged if we modify the set $S$ by replacing the sequence $\{4^n\}_{n \in \N}$ with any sequence $\{a_n\}_{n \in \N}$ such that for every $n\in \N, a_{n+1} > (2+ \epsilon) a_n$ for a fixed $\epsilon>0$. 
\end{remark}
\cref{4_8implies_not4_7} and \cref{Y_thing} imply \cref{0density_theorem} and so a negative answer to \cref{0densityquestion}.

\bibliographystyle{abbrv}
\bibliography{Results/For_Paper1/Drafts/Ref}
\end{document}